\documentclass{amsart}

\usepackage{amssymb,amsmath,graphicx}
\usepackage{color}

\newtoks\prt

\numberwithin{equation}{section}

\newtheorem{thm}{Theorem}[section]

\newtheorem{question}[thm]{Question}

\newtheorem{lemma}[thm]{Lemma}

\newtheorem{cor}[thm]{Corollary}

\newtheorem*{proclaim}{\the\prt}

\theoremstyle{definition}

\def\eqn#1$$#2$${\begin{equation}\label#1#2\end{equation}}

\def\fra{\mathfrak{A}}

\def\A{\mathcal A}
\def\B{\mathcal B}

\def\M{\mathcal M}

\def\ce{\mathbb C}

\def\cs{\operatorname{cs}}
\def\lin{Lindel\"of}

\def\co{\operatorname{co}}

\def\ep{\varepsilon}

\def\er{\mathbb R}

\def\ov{\overline}

\def \Bos {\operatorname{Bos}}

\def \ext {\operatorname{ext}}

\def\span{\operatorname{span}}

\def\wh{\widehat}
\def \reg {\partial _{\kern1pt\text{reg}}}

\def\di{\,\mbox{\rm d}}

\newcommand{\norm}[1]{\left\|#1\right\|}

\newcommand{\abs}[1]{\left| #1  \right|}

\renewcommand{\Re}{\operatorname{Re}}

\newcommand{\setsep}{;\,}

\newcommand{\eqdef}{\ \resizebox{4mm}{3mm}{\mbox{$\overset{\mathrm{def}}{=\joinrel=}$}}\ }
\newcommand{\zlomek}[2]{\mbox{$\frac{#1}{#2}$}}

\newcommand{\kolecko}{\resizebox{1mm}{1mm}{$\circ$}}
\newcommand{\acseg}{\mbox{$\,a,c\raisebox{0.5mm}{$^{\hspace{-5.65mm}\kolecko\hspace{-0.2mm}\raisebox{0.5mm}{\rule{5mm}{0.25pt}}\hspace{-0.2mm}\kolecko}$}$}}
\newcommand{\aeseg}{\mbox{$\,a,e\raisebox{0.5mm}{$^{\hspace{-5.65mm}\kolecko\hspace{-0.2mm}\raisebox{0.5mm}{\rule{5mm}{0.25pt}}\hspace{-0.2mm}\kolecko}$}$}\!}
\newcommand{\ccseg}{\mbox{$\,c_1,c_2\raisebox{0.5mm}{$^{\hspace{-8.3mm}\kolecko\hspace{-0.2mm}\raisebox{0.5mm}{\rule{7mm}{0.25pt}}\hspace{-0.2mm}\kolecko}$}$}}
\newcommand{\bbseg}{\mbox{$\,b_1,b_2\raisebox{1.1mm}{$^{\hspace{-8.6mm}\kolecko\hspace{-0.2mm}\raisebox{0.5mm}{\rule{7mm}{0.25pt}}\hspace{-0.2mm}\kolecko}$}$}}
\newcommand{\ceseg}{\mbox{$\,c,e\raisebox{0.5mm}{$^{\hspace{-5.45mm}\kolecko\hspace{-0.2mm}\raisebox{0.5mm}{\rule{5mm}{0.25pt}}\hspace{-0.2mm}\kolecko}$}$}}

\begin{document}

\title[Minimum principle]{The minimum principle for affine functions with the point of continuity property and  isomorphisms of spaces of continuous affine functions}

\author{Petr Dost\'al and Ji\v r\'\i\ Spurn\'y}

\address{Charles University\\
Faculty of Mathematics and Physics\\
Department of Probability and Mathematical Statistics \\
Sokolovsk\'{a} 83, 186 \ 75\\Praha 8, Czech Republic}
\email{dostal@karlin.mff.cuni.cz}

\address{Charles University\\
Faculty of Mathematics and Physics\\
Department of Mathematical Analysis \\
Sokolovsk\'{a} 83, 186 \ 75\\Praha 8, Czech Republic}
\email{spurny@karlin.mff.cuni.cz}



\thanks{The second author was supported by the Research grant  GA\v{C}R 17-00941S}

\begin{abstract}
Let  $X$ be a compact convex set and let $\ext X$ stand for the set of extreme points of $X$. We show that if $f\colon X\to\er$ is an affine function with the point of continuity property such that $f\le 0$ on $\ext X$, then $f\le 0$ on $X$.

As a corollary of this minimum principle  we obtain a generalization of a~theorem by H.B.~Cohen and C.H.~Chu by proving the following result. Let $X,Y$ be compact convex sets such that every extreme point of $X$ and $Y$ is a~weak peak point and let $T\colon \fra^c(X)\to \fra^c(Y)$ be an isomorphism such that $\norm{T}\cdot \norm{T^{-1}}<2$. Then $\ext X$ is homeomorphic to $\ext Y$.
\end{abstract}

\maketitle


\section{The minimum principle}

We  work within the framework of real vector spaces.
If $X$ is a compact convex set in a Hausdorff locally convex space and $\ext X$ is the set of all extreme points of $X$, the classical results assert that any semicontinuous  affine function $f\colon X\to\er$ satisfying $f\le 0$ on $\ext X$ is actually smaller or less then $0$ on $X$ (see e.g. \cite[Corollary~4.8 and Section~3.9]{lmns}). A generalization of this minimum principle can be found in \cite{sray} (see also \cite[Section 10.8]{lmns}). It is well known that any semicontinuous function $f\colon X\to \er$ has the \emph{point of continuity property}, i.e., $f|_F$ has a point of continuity for each $F\subseteq X$ closed (see \cite{koumou} or \cite[Theorem~A.121]{lmns}). The first goal of our paper is a proof of the following result.

\begin{thm}
\label{min}
Let $X$ be a compact convex set and $f\colon X\to \er$ be an affine function satisfying the point of continuity property. If $f\le 0$ on $\ext X$, then it follows that $f\le 0$ on $X$.
\end{thm}

This result prompts a question on validity of the minimum principle for strongly affine functions on a compact convex set $X$. Let us recall that any probability Radon measure $\mu\in\M^1(X)$ possesses its \emph{barycenter} $r(\mu)\in X$, i.e., the point satisfying
\begin{equation}
\label{e:barf}
h(r(\mu))=\int_X h(x)\di\mu(x),\quad h\in \fra^c(X).
\end{equation}
Here $\fra^c(X)$ denotes the space of all real continuous affine functions on $X$.
Then $f\colon X\to \er$ is said to be \emph{strongly affine}, if for each $\mu\in\M^1(X)$, $f$ is $\mu$-integrable and
\eqref{e:barf} holds for $f$, i.e., $f(r(\mu))=\int_X f\di\mu.$
Obviously, any strongly affine function is affine and moreover, it is bounded (see \cite[Satz~2.1]{krause}). It is well known that any affine function with the point of continuity property is strongly affine (see e.g. \cite[Chapter 14]{phelps-choquet} or \cite[Theorem~4.21]{lmns}). As shown by M.~Talagrand in \cite{tal3} (see also \cite[Theorem 12.65]{lmns}), the minimum principle does not hold for strongly affine functions. Nevertheless, the following question seems to be open.

\begin{question}
Let $f\colon X\to\er$ be a strongly affine Borel function on a compact convex set $X$ such that $f\le 0$ on $\ext X$. Does it follow that $f\le 0$ on $X$?
\end{question}

Of course, the answer to this question is yes in case $X$ is metrizable since then one can use that for any $x\in X$ there exists a measure $\mu\in\M^1(X)$ with $r(\mu)=x$ such that $\mu(\ext X)=1$ (see \cite[Corollary~I.4.9]{alfsen}, \cite[Chapter 3]{phelps-choquet} or \cite[Section~3.8]{lmns}). Also,
if $f$ is a~Baire strongly affine function, the answer is also yes since for any $x\in X$ there exists a~measure $\mu\in\M^1(X)$ with $r(\mu)=x$ such that $\mu(B)=1$ for each $B\supseteq \ext X$ Baire.

We recall a particular class of Borel functions that satisfy the point of continuity property. For a topological space $X$, let $\Bos(X)$ denote the algebra generated by closed (equivalently open) sets in $X$.  If $Y$ is another topological space, then a~mapping $f\colon X\to Y$ is \emph{of the first Borel class} if $f^{-1}(U)\in \left(\Bos(X)\right)_\sigma$ for any open set $U\subseteq Y$, i.e., $f^{-1}(U)$ is a countable union of sets from $\Bos(X)$. Obviously, any semicontinuous function is of the first Borel class. It is proved in \cite[Theorem~2.3]{koumou} that any real-valued function of the first Borel class on a compact space has the point of continuity property.

We start the proof of Theorem~\ref{min} by the following result from \cite{ggms}.

\begin{lemma}
\label{ggms}
Let $Y$ be a compact convex set and $f\colon Y\to\er$ be an affine function such that the set $C(f)$ of its points of continuity is dense in $Y$. Then $C(f)\cap \ext Y$ is dense in $\ext Y$.
\end{lemma}

\begin{proof}
See \cite[Lemma II.2]{ggms}.
\end{proof}

Below  we  use the following notation $\acseg\eqdef\co\{a,c\}\backslash \{a,c\}$,
where $a,c$ are points in a~vector space.
The following lemma is inspired by the proof of \cite[Proposition~3.1.1]{bourgin}.

\begin{lemma}
\label{triangle}
Let $E$ be a vector space, $a,b,c,c_1,c_2$ be distinct points in $E$ and $\Delta\eqdef\co \{a,c_1,c_2\}$. Let $f\colon \Delta\to \er$ be an affine function, $\eta\in \er$ and let
\begin{equation}\label{Fdef}
b\in F\eqdef\{x\in \Delta\setsep f(x)\ge \eta\},\quad a\notin F,\quad
c,c_1,c_2\in \{x\in \Delta\setsep f(x)>\eta\}.
\end{equation}
Let the following assumptions be satisfied.
\begin{itemize}
	\item [(A1)] The vectors $c_{1}-a$ and $c_1-c_2$ are linearly independent.
	\item [(A2)] We have $b\in\acseg$  and similarly  $c\in\ccseg.$
\end{itemize}
Then $b\notin \ext F$.
\end{lemma}

\begin{proof}
	First, we may assume that $\eta=0.$
	Otherwise, it is enough to replace $F$ by $\tilde F\eqdef F-\eta$ in the following proof.

(i) Let $f(b)>0.$
As $c\in F$ holds by assumption \eqref{Fdef},
it is enough to show that
\begin{equation}\label{ddef}
b\in\ceseg \ \ \ \textrm{ where }\ \ \ \
e\eqdef\zlomek{f(c)\,a-f(a)\,c}{f(c)-f(a)}\in F. 
\end{equation}
Since $f(a)<0<f(c)$ holds by assumption~\eqref{Fdef}, we get that $e$ in \eqref{ddef} is well defined and that
$e\in\co\{a,c\}\subseteq \Delta$ as $c\in\co\{c_1,c_2\}$ holds by assumption (A2).
Since $f$ is assumed to be an~affine function, we get
that $f(e)=0$, which ensures that $e\in F$
and also that $e\neq c$ as $f(c)>0$ holds by \eqref{Fdef}.
By assumption (A2) there exists $\alpha\in(0,1)$ such that
we have the first equality in
$$
b=\alpha a+(1-\alpha)\,c=\varepsilon e +(1-\varepsilon)\, c, \ \ \textrm{ where }\ \
\varepsilon\eqdef \alpha[1-{f(a)}/{f(c)}]>0.
$$
The second equality can be derived from the definition of $e$ in \eqref{ddef},
and to finish the first part of the proof, it is enough to show that $\varepsilon<1$, but it follows from the following relations
 $0<f(b)=\alpha f(a)+(1-\alpha)f(c)$ based on affinity of $f.$

 (ii) If $f(b)=0$, it is enough show that
 \begin{equation}\label{bdef}
 b\in\bbseg,\ \ \textrm{ where }\ \
 b_i\eqdef \zlomek{f(c_i)\,a-f(a)\,c_i}{f(c_i)-f(a)}\in F.
 \end{equation}
 Similarly as in (i) we would verify that $b_i$'s are well defined elements of $\Delta$ such that $f(b_i)=0$, which ensures $b_i\in F,i=1,2.$
 Further, it follows from assumption (A2) that there are positive values $\gamma_1,\gamma_2$ and $\alpha\eqdef 1-\gamma_1-\gamma_2>0$
 such that we have the first equality in
 $$
 b=\alpha a+\gamma_1 c_1+\gamma_2c_2=\beta_1b_1+\beta_2 b_2,\ \ \textrm{ where }\ \ \beta_i\eqdef\gamma_i[1-f(c_i)/f(a)]>0.
 $$
 The second equality can be verified just by computation using the definitions of $\alpha,b_i,\beta_i$'s and the following equality
 $
 0=f(b)=\alpha f(a)+\gamma_1 f(c_1)+\gamma_2 f(c_2).
 $
This equality is the one that should be used together with the definition of $\alpha$ in order to verify that $\beta_1+\beta_2=1.$
Thus, in order to verify \eqref{bdef} it is enough to show that $b_1\neq b_2,$
which  can be shown to be equivalent to the following inequality
$$
[f(a)-f(c_1)](c_1-c_2)\neq [f(c_1)-f(c_2)](a-c_1).
$$
This is obviously satisfied by assumption (A1), since $f(a)<0<f(c_1)$ holds by assumption~\eqref{Fdef}.
\end{proof}

\begin{proof}[{\bf Proof of Theorem~\ref{min}}]
	Let $f\colon X\to \er$ be an affine function with the point of continuity property and let $f\le 0$ on $\ext X$. Our aim is to show that also $f\le 0$ on~$X.$
	In other words, we assume that the following set is disjoint from $\ext X$
	$$F_\eta{\eqdef}\{x\in X\setsep f(x)\ge \eta\}$$
whenever $\eta\in(0,\infty),$ and we are going to show that it is empty.

To this end, assume the contrary, i.e.,~that there exists $\eta>0$ such that $F_\eta\neq\emptyset.$
Then $Y{\eqdef}\ov{F_\eta}$ is a~non-empty compact convex set, which ensures that $\ext Y\neq \emptyset.$
Since the set $C(f|_{Y})$ of points of continuity of $f|_Y$ is of the second category in $Y$ (see \cite[Theorem~2.3]{koumou}),
it is dense in $Y,$ and
we get from Lemma~\ref{ggms} that $C(f|_Y)\cap\ext Y$ is a~dense subset of $\ext Y\neq\emptyset.$
Then 
there has to exist a point
\begin{equation}\label{bin}
b\in C(f|_Y)\cap \ext Y.
\end{equation}
Since $f{|_Y}$ is continuous at $b$ and $b\in Y=\ov{F_\eta}$, {we get that} $f(b)\ge \eta$, i.e., $b\in F_\eta$.
As $F_\eta$ is disjoint from $\ext X$, we get that
$b$ is not an extreme point of $X.$ Then we can find $a,e\in X$ such that $b\in\aeseg$.
Since $b\in F_\eta\cap \ext Y\subseteq \ext F_\eta$ and $f$ is affine, either
$f(a)<\eta$ and $f(e)\ge \eta$
or vice versa. We assume that the former case holds. Then even
\begin{equation}\label{formercase}
f(a)<\eta<f(e)
\end{equation}
holds, since otherwise if $f(e)=\eta$, we would obtain that $f(b)<\eta,$ which is impossible as $b\in F_\eta.$
Put
\begin{equation}
\label{endpoint}
c\eqdef e+t(e-a)\in X, \ \textrm{ where } \
t\eqdef \max\{s\ge 0\setsep e+s(e-a)\in X\}.
\end{equation}
As $f$ is affine, we get from \eqref{formercase}, \eqref{endpoint} that
$f(c)=f(e)+t[f(e)-f(a)]\geq f(e)>\eta.$
Hence, $c\in F_\eta.$
If $c\in \ext X,$ we have a~contradiction with our assumption that $F_\eta\cap\ext X=\emptyset.$
So let us assume the contrary, i.e.,~that $c\in\ccseg$ for some $c_1,c_2\in X$.
We may assume that $c_1,c_2$ are chosen so that $f(c_i)>\eta$, $ i=1,2$, since otherwise we would consider $\tilde c_i\eqdef c+\varepsilon(c_i-c)$ for $\varepsilon>0$ small enough instead.
By the choice of $c$ in \eqref{endpoint}, the vectors $c-a$ and $c_1-c_2$ are linearly independent,
which obviously means that the same holds with $c$ replaced by $c_1.$
Now we are at the situation of Lemma~\ref{triangle}.
Thus it follows that $b\in F\backslash\ext F,$ where $ F\eqdef F_\eta\cap \Delta$ and $\Delta\eqdef\co\{a, c_1,c_2\}$.
Then $b$ cannot be an~extreme point of any superset of $ F.$
In particular, $b\notin \ext Y$, where $Y=\ov{F_\eta}\supseteq F_\eta\supseteq F$, and we have a~contradiction with \eqref{bin},
which finishes the proof.
\end{proof}

Using the Hahn-Banach theorem we can obtain the following corollary.

\begin{cor}
\label{min-vector}
Let $X$ be a compact convex set.
\begin{itemize}
\item [(a)] If $F$ is a locally convex space and $f\colon X\to F$ is affine with the point of continuity property, then $f(X)\subseteq \ov{\co} f(\ext X)$.
\item [(b)] If $f\colon X\to \ce$ is affine and has the point of continuity property, then $\sup_{x\in X} \abs{f(x)}=\sup_{x\in \ext X} \abs{f(x)}$.
\end{itemize}
\end{cor}

\begin{proof}
(a) We assume that there exists $x\in X$ such that $f(x)\notin \ov{\co} f(\ext X)$. By the Hahn-Banach theorem we can find $\tau\in F^*$ such that
\[
	\tau(f(x))>\eta\eqdef\sup \{\tau(z)\setsep z\in  \ov{\co} f(\ext X)\}\in\mathbb R.
\]
It straightforwardly follows from the definition that $\tau\circ f$ has the point of continuity property and the same holds for the function $\tau\circ f-\eta$ attaining values only in $(-\infty,0]$ on $\ext X.$
Then we obtain by Theorem~\ref{min}
\[
	\tau(f(x))=(\tau\circ f)(x)\le {\eta}
<\tau(f(x)),
\]
i.e., a contradiction.

(b)
We identify $\ce$ with $\er^2$. By (a)  we have for each $x\in X$ that
\[
	f(x)\in \ov{\co} f(\ext X)\subseteq {K_\eta\eqdef\{\lambda\in\ce\setsep \abs{\lambda}\le \eta\},
	\ \textrm{ where }\ \eta\eqdef \sup_{x\in{\ext} X} \abs{f(x)},}
\]
{since $K_\eta$ is obviously a~closed convex superset of $f(\ext X).$}
Thus $\abs{f(x)}\le \eta$, which finishes the proof.
\end{proof}

\section{A generalization of the Cohen-Chu theorem}

The aim of this part is a proof of a generalization of the result by C.\,H.~Chu and H.\,B.~Cohen in the spirit of the Banach-Stone theorem.
They proved  the following theorem (see \cite{cc}):

\emph{Let $X$ and $Y$ be compact convex sets and let $T\colon \fra^c(X) \to \fra^c(Y)$ be an isomorphism satisfying
$\|T\|\cdot\|T^{-1}\|<2$. If
\begin{itemize}
\item $X$ and $Y$ are metrizable and each point of $\ext X$ and $\ext Y$ is a weak peak point, or
\item the sets $\ext X$ and $\ext Y$  are closed and each extreme point of $X$ and $Y$ is a split face,
\end{itemize}
then the sets $\ext X$ and $\ext Y$ are homeomorphic.}

For a set $F\subseteq X$, the {\textit{complementary set}} $F^{\cs}$ is defined as the union of all faces of $X$ disjoint from $F$. A face $F$ of $X$ is said to be
a~\textit{split face} if its complementary set $F^{\cs}$ is convex (and hence a face, see \cite[p.\,132]{alfsen}) and every point in $X \setminus (F \cup F^{\cs})$ can be uniquely represented as a convex combination of a point in $F$ and a point in $F^{\cs}$.

We call $x\in \ext X$ a~\textit{weak peak point} if for each $\ep \in (0,1)$ and an open neighborhood $U$ of $x$ there exists $h \in \fra^c(X)$ such that
$$
\|h\| \leq 1,\  h(x) > 1 - \ep\ \textrm{ and }\ |h| < \ep\ \textrm{  on }\ \ext X \setminus U.
$$
Let us also recall that {if $x$ is} a~weak peak point of a compact convex set $X$, {then $\{x\}$} is a~split face and the converse holds if $\ext X$ is closed; see \cite[Proposition~1]{cc}.

We refer the reader to \cite[pp.\,72,\,73,\,75]{cc} for notions of the theory of compact convex sets (see also \cite[Section~4.3]{lmns}). We just mention that $X$ can be embedded to $(\fra^c(X))^*$ via the evaluation mapping $\phi\colon X\to (\fra^c(X))^*$ defined as $\phi(x)(f)=f(x)$, $f\in\fra^c(X)$, $x\in X$.
The dual unit ball $B_{(\fra^c(X))^*}$ equals the convex hull $\co (X\cup -X)$ and $(\fra^c(X))^*$ coincides with $\span X$, the linear span of $X$.
Further, any affine bounded function $f$ on $X$ has the unique extension to $\span X$, and this provides an identification of $(\fra^c(X))^{**}$ with the space $\fra^b(X)$ of all bounded affine functions {on~$X$.}

We use Theorem~\ref{min} to show the following theorem.

\begin{thm}
\label{cochu}
Let $X,Y$ be compact convex sets such that every extreme point of $X$ and $Y$ is a weak peak point and let ${T}\colon \fra^c(X)\to \fra^c(Y)$ be an isomorphism with $\norm{T}\cdot \norm{T^{-1}}<2$. Then $\ext X$ is homeomorphic to $\ext Y$.
\end{thm}

Example~1 on \cite[p.\,83]{cc} shows that Theorem~\ref{cochu} need not hold even for compact convex sets in finite dimensional spaces if we omit the assumption that extreme points are weak peak points. An example due to H.\,U.~Hess (see \cite{hess}) shows that \emph{for every $\ep>0$ there exist metrizable simplices $X$, $Y$ and an isomorphism $T\colon\fra^c(X)\to\fra^c(Y)$ such that $\|T\|\cdot\|T^{-1}\|<1+\ep$ and $\ext X$ is not homeomorphic to $\ext Y$.} Also, the bound $2$ is optimal by a result of H.\,B.~Cohen (see \cite{cohen2}). We also mention paper~\cite{lusppams} where Theorem~\ref{cochu} is proved under condition that the sets of extreme points are \lin.

\begin{proof}[{\bf Proof of Theorem~\ref{cochu}}]
We follow 
the proof of \cite[Theorem 1.1]{lusppams}. The 
main difference is that we use Theorem~\ref{min} instead of \cite[Lemma~2.1]{lusppams} and thus we need to verify its assumptions in Claim 2.
Let $T\colon \fra^c(X)\to \fra^c(Y)$ be an~isomorphism satisfying $\|T\|\cdot\|T^{-1}\|<2$. We may assume that
there exists $1<c'<2$~such that
\begin{equation}\label{TT}
	||T||<2\quad\text{and}\quad  \|Tf\|\geq c'\|f\| \ \textrm{ for all }\  f\in\fra^c(X).
	\end{equation}
Otherwise, we would find $1<c'<2$ such that $\|T\|\cdot \|T^{-1}\|<\frac{2}{c'}<2$ and
consider $\tilde T\eqdef c'\|T^{-1}\| T$ instead; see \cite[p.\,76]{cc}. Further, we fix $1<c<c'.$

\prt{Claim 1}

\begin{proclaim}
For any $f\in \fra^b(X)$ and $g\in\fra^b(Y)$ non-zero, $\|T^{**}f\|>c\|f\|$ and $\|(T^{-1})^{**}g\|>\frac 12\|g\|$.
\end{proclaim}

\begin{proof}[Proof of Claim 1]
See \cite[Proof of Claim 1]{lusppams}.
\end{proof}

If $x\in\ext X$, we recall that $(\fra^c(X))^*=\span \{x\}\oplus_{\ell^1} \span \{x\}^{\cs}$ because $\{x\}$ is a~split face; see \cite[p.\,72]{cc}.
Hence, given $y\in Y$, following \cite[p.\,76]{cc} we can write
\begin{equation}
\label{rozklad}
T^*y=\lambda x+\mu\quad\text{for some }\lambda\in\er\text{ and }\mu\in \span\{x\}^{\cs}.
\end{equation}
Similarly as in \cite[p.\,77]{cc}, for $y\in Y$ satisfying \eqref{rozklad}, we have that
\begin{equation}\label{implikace}
|\lambda|>c\ \ \Rightarrow \ \
	\|\mu\|=||T^*y||-|\lambda|<2-c.
\end{equation}
Given $x\in \ext X$, we denote by $\chi_{\{x\}}$ the characteristic function of the set $\{x\}$. Then the upper envelope function $h_x\eqdef\widehat{\chi}_{\{x\}}$, defined as
\[
\widehat{\chi}_{\{x\}}(z) \eqdef \inf \{ h(z): h \in \fra^c(X), h > \chi_{\{x\}} \} \quad \text{for } z \in X,
\]
is upper semicontinuous and affine (see \cite[p.\,73]{cc}), and thus strongly affine (see \cite[Theorem~1.6.1(ix)]{AE}).
Futher, we note (see \cite[p.\,77]{cc}) that
\begin{equation}\label{hxmu}
	\{x\}=h_x^{-1}\{1\}\ \ \textrm{ and }\ \
	\{x\}^{\cs}=h_x^{-1}\{0\}.
\end{equation}

\prt{Claim 2}

\begin{proclaim}
For any $x\in \ext X$, $T^{**}h_x$ is a strongly affine function of the first Borel class and thus it has the point of continuity property.
\end{proclaim}

\begin{proof}[Proof of Claim 2]
Since $T\colon \fra^c(X)\to \fra^c(Y)$, we have $T^*\colon \span Y\to \span X$. If $f\in\fra^b(X)$ and $\wh f$ is the linear extension of $f$ to $\span X$, then
$T^{**}f=\wh f\circ T^*$.
Since $\|T\|<2$,
\[
T^*Y\subseteq 2 B_{(\fra^c(X))^*}=\co (2X\cup -2X).
\]

The function $f\eqdef h_x$, being upper semicontinuous and affine, is strongly affine on $X$. The sets $2X$ and $-2X$ are affinely homeomorphic to $X$, and hence $\wh f$ is strongly affine on both of them. By \cite[Lemma~2.4(b)]{spurnytams}, $\wh{f}$ is strongly affine on
\[
2B_{(\fra^c(X))^*}=\co (2X\cup -2X).
\]
Since $Y$ is affinely homeomorphic to $T^*Y$ and $T^{**}f=\wh{f}\circ T^*$, we obtain that $T^{**}f$ is strongly affine on $Y$.

Further, $h_x$ is upper semicontinuous on $X$ and thus it is of the first Borel class on $X$. Since $2X$ and $-2X$ are affinely homeomorphic to $X$, $\wh{f}$ is of the first Borel class on $2X\cup -2X$.
Now we can use \cite[Theorem 3.5(b)]{lusp} to conclude that $\wh{f}$ is of the first Borel class on $2B_{(\fra^c(X))^*}=\co (2X\cup -2X)$. As above we obtain that $T^{**}h_x$ is of the first Borel class on $Y$.
The final statement now follows from \cite[Theorem~2.3]{koumou}.
\end{proof}

\noindent
{Similarly as in \cite[p.\,77]{cc} we consider mappings $\rho\colon\ext Y\to \ext X,\tau\colon\ext X\to \ext Y$  defined as follows
	\begin{equation}\label{tilderho}
	\tilde \rho\eqdef\{(y,x)\in \ext Y\times \ext X\setsep |T^{**} h_x(y)|>c\},
\end{equation}
	\begin{equation}
	\tilde \tau\eqdef\{(x,y)\in \ext X\times \ext Y\setsep |(T^{-1})^{**} h_y(x)|>\zlomek12\}.
\end{equation}
}
By \cite[p.\,77]{cc}, {$\tilde\rho$ is a~mapping and we denote its domain as $\tilde Y\eqdef\textrm{dom}(\tilde\rho).$
Analogously, we would get that also $\tilde\tau$ is a~mapping and we put $\tilde X\eqdef\textrm{dom}(\tilde\tau).$}

Note that if $x\in\ext X,y\in\ext Y\subseteq Y$ and $\lambda$ are as in \eqref{rozklad}, then for the linear extension $\wh{h}_{x}$ of $h_x$ on $\span X=(\fra^c(X))^*$ we have
\begin{equation}\label{Thxylambda}
T^{**}h_{x}(y)
=\wh{h}_x(T^*y)=\wh{h}_{x}(\lambda x+\mu) =\lambda h_x(x)+h_x(\mu)=\lambda
\end{equation}
 as $\wh{h}_{x}$ is linear and
$h_x(x)=1$ and $h_x(\mu)=0$ hold by \eqref{hxmu}.

\prt{Claim 3}
\begin{proclaim}
	The mappings ${\tilde\rho}\colon \tilde Y\to \ext X$ and ${\tilde\tau}\colon\tilde  X\to \ext Y$ are surjective.
\end{proclaim}

\begin{proof}[Proof of Claim 3]
Let $x\in \ext X$ be given and assume that $|T^{**}h_x(y)|\leq c$ for all $y\in \ext Y$. By Theorem~\ref{min} and Claim~2, $|T^{**}h_x|\leq c$ on $Y$. Then
\[
c\geq \|T^{**}h_x\|>c\|h_x\|=c
\]
gives a contradiction. Hence ${\tilde\rho}$ is surjective.
Analogously, using the second part of Claim~1 we would obtain that ${\tilde\tau}$ is surjective.
\end{proof}

The following claim is essentially Lemma~6 of \cite{cc} and Claim 4 in \cite{lusppams}. However, we recall its proof since it uses Theorem~\ref{min}.
\prt{Claim 4}
\begin{proclaim}
We have $\tilde X=\ext X$ and $\tilde Y=\ext Y$ and, for any $x\in \ext X$ and $y\in\ext Y$, $\tilde\rho(\tilde\tau(x))=x$ and $\tilde\tau(\tilde\rho(y))=y$.
\end{proclaim}

\begin{proof}[Proof of Claim 4]
	We will show that $(\tilde\rho(\tilde y),\tilde y)\in\tilde\tau$ holds for any $\tilde y\in\tilde  Y$, i.e.,
	\begin{equation}
		\label{rho} |(T^{-1})^{**}h_{\tilde y}({\tilde\rho}(\tilde y))|>\zlomek12.
\end{equation}
{By Claim~1, $\|(T^{-1})^{**}h_{\tilde y}\|>\zlomek12\,\|h_{\tilde y}\|=\zlomek12.$ Then}
Claim~2 and
Theorem~\ref{min} yield
\begin{equation}\label{ddefT}
	d\eqdef \sup_{\tilde x\in\ext X} |(T^{-1})^{**}h_{\tilde y}(\tilde x)|=\sup_{\tilde x\in X} |(T^{-1})^{**}h_{\tilde y}(\tilde x)|=\|(T^{-1})^{**}h_{\tilde y}\|>\zlomek12.
\end{equation}
Since $c > 1$, we have $d>\max\{\frac{d}{c},\zlomek12\}$. Hence, there exists $x\in\ext X$ such that
\begin{equation}\label{xyintau}
	|(T^{-1})^{**}h_{\tilde y}(x)|>\max\{\zlomek{d}{c},\zlomek12\}\geq \zlomek12, \ \textrm{ i.e. } (x,\tilde y)\in \tilde\tau.
\end{equation}
Let us assume that \eqref{rho} does not hold. Then ${\tilde\rho}(\tilde y)\neq x$,
and by Claim~3 we can find $y\in\tilde Y$ with ${\tilde\rho}(y)=x$.
Then $y\in \{\tilde y\}^{\cs}$, and thus $h_{\tilde y}(y)=0$.
Since $x\in\ext X$ and $y\in Y$, we can use decomposition \eqref{rozklad} in order to get that
\begin{equation}
	\label{nula} 0=h_{\tilde y}(y)= (T^{-1})^{**}h_{\tilde y}(T^*y)=(T^{-1})^{**}h_{\tilde y}(\lambda x )+(T^{-1})^{**}h_{\tilde y}(\mu).
\end{equation}
Since $\lambda=T^{**}h_{x}(y)$ holds by \eqref{Thxylambda}, and as $(y,x)\in\tilde\rho$, we get by \eqref{tilderho} that $|\lambda|>c$.
Then we get from \eqref{implikace}, \eqref{ddefT}, \eqref{xyintau} and \eqref{nula} that
\[
\aligned
d&<|\lambda|\zlomek{d}{c}<|\lambda|\cdot|(T^{-1})^{**}h_{\tilde y}(x)|=|(T^{-1})^{**}h_{\tilde y}(\lambda x )|\\
&=|(T^{-1})^{**}h_{\tilde y}(\mu)|\leq d\|\mu\|<d(2-c)<d.
\endaligned
\]
This is a~contradiction with assumption that \eqref{rho} does not hold, hence \eqref{rho} holds,
and we have that
\begin{equation}\label{inverz}
\tilde \tau(\tilde \rho(y))=y, \quad y\in\tilde Y.
\end{equation}
Now, let $x\in \ext X$ be given. By Claim~3 there exists $\tilde y\in\tilde Y$ with $\tilde\rho(\tilde y)=x$.
Then we get from \eqref{inverz} that $\tilde \tau(x)=\tilde y$, which ensures that $x\in \tilde X\subseteq \ext X$
and finally that also $\tilde X=\ext X$.

Let $y\in \ext Y$ be given. From Claim~3 we obtain first that there exists $x\in \tilde X=\ext X$ such that $\tilde \tau(x)=y$
and then  that there exists $\tilde y\in\tilde Y$ such that $\tilde \rho(\tilde y)=x$.
Then we get  from \eqref{inverz} that
$
y=\tilde \tau(x)=\tilde \tau(\tilde \rho(\tilde y))=\tilde y\in\tilde Y\subseteq\ext Y,
$
and finally that $\tilde Y=\ext Y.$

If $x\in \ext X=\tilde X$, it is enough to use Claim~3 once again and property \eqref{inverz} in order to get that
$
\tilde\rho(\tilde\tau(x))=\tilde\rho(\tilde\tau(\tilde\rho(y)))=\tilde\rho(y)=x
$
holds for some $y\in\ext Y.$
\end{proof}

By the proof of Theorem~7 on p.\,78 in \cite{cc}, the mappings $\tilde\rho$ and $\tilde\tau$ are continuous ($\tilde{\rho}$ and $\tilde{\tau}$ are denoted as $\rho$ and $\tau$ in \cite{cc}).
This finishes the proof of Theorem~\ref{cochu}.
\end{proof}

As in \cite[Corollaries 13 and 14]{cc} we obtain the following corollary.

\begin{cor}
Let $\A$ and $\B$ be function
algebras, and let $T\colon \Re\A\to \Re\B$ be an isomorphism satisfying $\|T\|\cdot
\|T^{-1}\|<2$. Then the Choquet boundaries of $\A$ and $\B$ are homeomorphic.
\end{cor}


\end{document}